\documentclass[MSNbibl,number,citesort,seceqn,dvips]{arxbj}
\usepackage{mathrsfs}

\aid{0}
\volume{20}
\issue{1}
\pubyear{2014}
\firstpage{190}
\lastpage{206}
\doi{10.3150/12-BEJ481} 

\makeatletter
\newcommand{\rrvert}{\vert}
\newcommand{\llvert}{\vert}
\newtheorem{theorem}{Theorem}
\newtheorem{lemma}{Lemma}

\newtheorem{corollary}{Corollary}

\newproclaim{defi}{Definition}
\newremark{remark}{Remark}
\newcommand{\eqref}[1]{(\ref{#1})}

\newcommand{\eps}{\varepsilon}
\newcommand{\cadlag}{c\`adl\`ag}
\newcommand{\F}{\mathscr{F}}
\newcommand{\nbd}{-}
\newcommand{\sa}{$\sigma$\nbd field}
\newcommand{\re}{\mathbb{R}}
\newcommand{\p}{\bb{P}}
\newcommand{\defin}[1]{$\mathrm{#1}$}
\newcommand{\bb}[1]{\mathbb{#1}}
\newcommand{\mc}[1]{\mathscr{#1}}
\newcommand{\q}{\bb{Q}}
\newcommand{\se}{\bb{E}}
\newcommand{\si}{\mathbf{1}}
\newcommand{\cond}[2]{\vphantom{#2}#1\ | #2}

\makeatother

\begin{document}
\begin{frontmatter}

\title{Bridges of L\'evy processes conditioned to stay positive}
\runtitle{Bridges of L\'evy processes conditioned to stay positive}

\begin{aug}
\author{\fnms{Ger\'onimo} \snm{Uribe Bravo}\corref{}\ead[label=e1]{geronimo@matem.unam.mx}}
\runauthor{G. Uribe Bravo} 
\address{Instituto de Matem\'aticas, Universidad Nacional Aut\'onoma
de M\'exico, \'Area de la Investigaci\'on Cient\'ifica,
Circuito Exterior, Ciudad Universitaria, Coyoac\'an, 04510, M\'exico,
D.F.\\ \printead{e1}}
\end{aug}

\received{\smonth{3} \syear{2012}}
\revised{\smonth{9} \syear{2012}}

%
\begin{abstract}
We consider Kallenberg's hypothesis on the characteristic function of a
L\'evy
process and show that it allows the construction of weakly continuous
bridges of
the L\'evy process conditioned to stay positive. We therefore provide a notion
of normalized excursions L\'evy processes above their cumulative
minimum. Our
main contribution is the construction of a continuous version of the transition
density of the L\'evy process conditioned to stay positive by using the weakly
continuous bridges of the L\'evy process itself. For this, we rely on a method
due to Hunt which had only been shown to provide upper semi-continuous versions.
Using the bridges of the conditioned L\'evy process, the
Durrett--Iglehart theorem
stating that the Brownian bridge from $0$ to $0$ conditioned to remain
above $-\eps$
converges weakly to the Brownian excursion as $\eps\to0$, is extended
to L\'evy processes.
We also extend the Denisov decomposition of Brownian motion to L\'evy
processes and their bridges, as well as
Vervaat's classical result stating the equivalence in law of the
Vervaat transform of a Brownian bridge and the normalized Brownian excursion.
\end{abstract}

%
\pdfkeywords{L\'evy processes, Markovian bridges, Vervaat transformation}
\end{frontmatter}

\section{Introduction and statement of the results}
Our discussion will use the canonical setup: $X= ( X_t
)_{t\geq0}$
denotes the canonical process on the Skorohod space of \cadlag\
trajectories, $\F$ denotes \sa\ generated by $X$ (also written
$\sigma  ( X_s,s\geq0 )$), $ ( \F_t,t\geq0
)$ is the canonical filtration,
where $\F_t=\sigma  ( X_s,s\leq t )$ and $\theta_t,t\geq
0$, are the shift
operators given by $\theta_t  ( \omega )_s=\omega_{t+s}$.
Emphasis is
placed on the various probability measures considered.

Focus will be placed on two special (Markovian) families of probability
measures, denoted $ ( \p_x,x\in\re )$ and $ ( \p_x^\uparrow ,x\geq0 )$. The probability measure $\p_0$
corresponds to the law of a
L\'evy process: under $\p_0$ the canonical process has independent and
stationary increments and starts at $0$. Then $\p_x$ is simply the law
of $x+X$ under $\p_0$, and under each $\p_x$ the canonical process is
Markov and the conditional law of $ ( X_{t+s},s\geq0 )$
given $\F_t$ is $\p_{X_t}$. (Collections of probability measures on Skorohod
space satisfying the latter property are termed Markovian families.) We
also make use of the dual L\'evy process by letting $\hat\p_x$ denote
the law of $x-X$ under $\p_0$. Associated to $\p_x,x\in\re$, $\p_x^\uparrow$ can be interpreted as the law of the L\'evy process
conditioned to stay positive; as this event can have probability zero,
the precise definition of $\p_x^\uparrow$ can be described as follows:
a~L\'evy process
conditioned to stay positive is the (weak) limit of $X$ conditioned to
stay positive until an independent exponential $T_\alpha$ of parameter
$\alpha$ as $\alpha\to0$ (cf. Chaumont and Doney \cite{MR2164035}, Proposition 1). It is
actually simpler to actually construct L\'evy processes conditioned to
stay positive by a Doob transformation and justifying this passage to
the limit afterwards, as recalled in Section \ref
{ConditionedLPAndMeanderSection}.

Under very general conditions, given a Markovian family of probability
laws like $ ( \p_x,x\in\re )$, one can construct weakly continuous
versions of the conditional laws of $ ( X_s,s\leq t )$ under
$\p_x$
given $X_t=y$. They are termed bridges of $\p_x$ between $x$ and $y$ of
length $t$ and usually denoted $\p_{x,y}^t$. In Section \ref
{WeaklyContinuousBridgesSection}, we review the construction of these
bridges from Chaumont and Uribe~Bravo~\cite{MR2789508}. Our first result is to show that one can
apply this general recipe to the laws $\p_x^\uparrow$. To this end, we
impose two conditions on the L\'evy process.
\begin{longlist}[(K)]
\item[(K)] Under $\p_0$ and for any $t>0$, 
$
\int\llvert \se_0  ( \mathrm{e}^{\mathrm{i} u X_t} )\rrvert   \,\mathrm{d}u<\infty.
$
\item[(R)] $0$ is regular for both half-lines $(-\infty,0)$ and
$(0,\infty)$.
\end{longlist}
Assumption \defin{(K)} was introduced by Kallenberg as a means of
imposing the existence of densities for the law of $X_t$ for any $t>0$
which posses good properties (in particular continuity). A~construction
of L\'evy process bridges under hypothesis \defin{(K)} was first
accomplished in Kallenberg~\cite{MR628873} by means of convergence criteria for
processes with exchangeable increments. This construction is retaken as
an example of the general construction of Markovian bridges in Chaumont and Uribe~Bravo \cite
{MR2789508}.
%
\begin{theorem}
\label{WeaklyContinuousBridgeLPCSP}
Under \defin{(K)} and \defin{(R)}, we can construct bridges $\p_{x,y}^{\uparrow, t}$ of $\p_x^\uparrow$ for any $x,y\geq0$ and $t>0$
and that they are weakly continuous as functions of $x$ and $y$.
\end{theorem}

Theorem \ref{WeaklyContinuousBridgeLPCSP} presents another example of
the applicability of Theorem 1 in Chaumont and Uribe~Bravo \cite{MR2789508}, and the proof of
the former consists on verifying the technical hypotheses in the
latter. These technical hypotheses are basically: the existence of a
continuous and positive version of the densities of $X_t$ under $\p_x^{\uparrow}$. For nonzero starting states, we will inherit absolute
continuity from that of the L\'evy process killed upon becoming
negative (in Lemma \ref{ExistenceOfDensitiesForLPCSP}) and the later
can be studied by a technique inspired from Hunt \cite{MR0079377} for the
Brownian case, in which a transition density for the killed L\'evy
process is obtained from a transition density of the L\'evy process
using its bridges. (Cf. Equation \eqref{HuntsFormulaForDensity} and
Lemma \ref{ExistenceOfDensitiesForKilledProcess}.) Hunt's technique has
typically allowed only the construction of lower semicontinuous
versions of the density, but with weakly continuous bridges one can
show that Hunt's density is actually continuous. This is one possible
application of the existence of weakly continuous Markovian bridges.
Another problem is then to characterize the points at which the density
is positive. Hunt does this for Brownian motion and the result has been
extended to (multidimensional) stable L\'evy processes in the symmetric
case by Chen and Song \cite{MR1473631}, Theorem~2.4 and in the asymmetric case by
Vondra{\v{c}}ek \cite{MR1918106}, Theorem~3.2. We study positivity of the density by
exploiting the cyclic exchangeability property of L\'evy processes,
following Knight \cite{MR1417982}.

Recall that when $\p_0$ is the law of Brownian motion, $\p_{x,y}^t$ is
the law of the Brownian bridge between $x$ and $y$ of length $t$ and
the corresponding law $\p_{0,0}^{\uparrow,t}$ is the law of a Brownian
excursion of length $t$. %
In this context, our next result is an extension to L\'evy processes of
the classical result of Durrett, Iglehart and
Miller \cite{MR0436353}, which covers the Brownian case.

\begin{corollary}
\label{DIMCorollary}
The conditional law of $ ( X_s,s\leq t )$ under $\p_{0,0}^t$ given
$\underline X_t>-\eps$, where
\[
\underline X_t=\inf_{s\leq t}X_s,
\]
converges weakly, as $\eps\to0$ to $\p_{0,0}^{\uparrow,t}$.
\end{corollary}

The Brownian case of Corollary \ref{DIMCorollary} was first proved by
Durrett, Iglehart and
Miller \cite{MR0436353} by showing the convergence of finite-dimensional
distributions and then tightness, which follows from explicit
computations with Brownian densities. Another proof for the Brownian
case was given by Blumenthal \cite{MR704566} this time using rescaling, random
time change and simple infinitesimal generator computations. For us,
Corollary \ref{DIMCorollary} is a simple consequence of Theorem \ref
{WeaklyContinuousBridgeLPCSP}.

We now present a generalization of a decomposition of the Brownian
trajectory at the time it reaches its minimum on a given interval due
to Denisov \cite{MR726906}.

Let $\rho_t$ be the first time that $X$ reaches its minimum on the
interval $[0,t]$. Consider the pre- and post-minimum processes on the
interval $[0,t]$ given by:
\[
{X}^{\leftarrow}_s=X_{ ( \rho_t-s )^+-}-\underline X_t\quad
\mbox {and} \quad{X}^{\rightarrow}_s=X_{ ( \rho_t+s )\wedge t}-\underline
X_t
\]
defined for $s\geq0$, where $X_{s-}$ is the left limit of $X$ at $s$.

A L\'evy meander of length $t$ (following Chaumont and Doney \cite{MR2663630}) is the weak
limit as $\eps\to0$ of $X$ conditioned to remain above $-\eps$ on
$[0,t]$ under $\p_0$. L\'evy meanders can also be characterized by an
absolute continuity relationship with L\'evy processes conditioned to
stay positive as recalled in Section \ref
{ConditionedLPAndMeanderSection}. Denote by $\p^{\mathrm{me},t}$ the law of a
meander of length $t$ and by $\hat\p^{\mathrm{me},t}$ the meander of the dual
L\'
evy process.

\begin{theorem}
\label{DenisovTheorem}
Assume conditions \defin{(K)} and \defin{(R)}.
Under $\p_{0}$, the conditional law of $ ( {X}^{\leftarrow
},{X}^{\rightarrow} )$
given $\rho_t$ is $\hat\p^{\mathrm{me},\rho_t}\otimes\p^{\mathrm{me},t-\rho_t}$.
\end{theorem}
The previous result is a consequence of results in Chaumont and Doney \cite{MR2164035}. It
is our stepping stone on the way to our generalization of Vervaat's
relationship between the Brownian bridge and the normalized Brownian
excursion. This extension requires the following conditioned version of
Theorem \ref{DenisovTheorem}.
%
\begin{theorem}
\label{DenisovForBridgesTheorem}
A regular conditional distribution of $ ( X^\leftarrow
,X^\rightarrow  )$ given $\rho_t=s$ and $-\underline X_t=y$
under $\p_{0,0}^t$ is
$\hat
\p^{\uparrow,s}_{0,y}\otimes\p^{\uparrow,t-s}_{0,y}$.
\end{theorem}

We finally turn to an extension of the classical relationship between
the Brownian bridge between $0$ and $0$ and the Brownian excursion of
the same length.

\begin{theorem}
\label{VervaatTheorem}
Define the Vervaat transformation $V$ of $X$ on $[0,t]$ by
\[
V_s=X_{  ( \rho_t+s )\bmod t}-\underline X_t.
\]
Under \defin{(K)} and \defin{(R)}, the law of $V$ under $\p_{0,0}^t$ is
$\p_{0,0}^{\uparrow,t}$.
\end{theorem}

Theorem \ref{VervaatTheorem} was found by Vervaat \cite{MR515820} for Brownian
motion and proved there using approximation by a simple random walk.
Biane \cite{MR838369} gives a proof using excursion theory for Brownian
motion. Then Chaumont \cite{MR1465814} gave a definition of normalized stable
excursion and proved Theorem \ref{VervaatTheorem} in the case of stable
L\'evy processes, again using excursion theory. This extension of
Vervaat's theorem is the closest to the one in this work. Miermont \cite
{MR1844511} gives a version of Theorem \ref{VervaatTheorem} for
spectrally positive L\'evy processes in the context of the intensity
measures for excursions above the cumulative minimum, with an explicit
link with the L\'evy process conditioned to stay positive. Finally,
Fourati \cite{MR2139029} gives an abstract version of Theorem \ref
{VervaatTheorem} for L\'evy processes, again as a relation between two
$\sigma$-finite measures which can be though of as bridges of random
length, although there is no explicit link with L\'evy processes
conditioned to stay positive. After establishing this link, Theorem
\ref{VervaatTheorem} would follow from the theory developed in Fourati \cite
{MR2139029} using regularity results for bridges (like weak continuity)
in order to condition by the length. Instead of that, we propose a
direct proof.

The paper is organized as follows. In Sections \ref
{WeaklyContinuousBridgesSection} and \ref{LPSection}, we review the
construction of Markovian and L\'evy bridges of Chaumont and Uribe~Bravo \cite{MR2789508}. In
Section \ref{ConditionedLPAndMeanderSection}, we define, following
Chaumont and Doney \cite
{MR2164035}, L\'evy processes conditioned to stay positive and
meanders. Section \ref{BridgesOfConditionedLPSection} is devoted to the
construction of bridges of L\'evy processes conditioned to stay
positive, where we prove Theorem \ref{WeaklyContinuousBridgeLPCSP} and
Corollary \ref{DIMCorollary}. In Section \ref{DenisovSection}, we
consider extensions and consequences of Denisov's theorem, proving in
particular Theorems \ref{DenisovTheorem} and \ref
{DenisovForBridgesTheorem}. Finally, in Section \ref{VervaatSection},
we prove our extension of Vervaat's theorem, which is Theorem \ref
{VervaatTheorem}.

\section{Weakly continuous bridges of Markov processes}
\label{WeaklyContinuousBridgesSection}
Let $\mathbf{P}_x$ be the law of a Feller process which starts at $x$
which is an element of a polish space $S$ (for us either $\re$,
$(0,\infty)$, or $[0,\infty)$).
Suppose $P$ is its semigroup and assume that:
\begin{longlist}[(AC)]
\item[(AC)]There is a $\sigma$-finite measure $\mu$ and a function
$h_t  ( x,\cdot )$ such that
\[
P_tf ( x )=\int f ( y ) h_t ( x,y ) \mu ( \mathrm{d}y ).
\]
\item[(C)]The function $ ( s,x,y )\mapsto h_s  (
x,y )$ is continuous.
\item[(CK)]The Chapman--Kolmogorov equations
\[
h_{s+t} ( x,z )=\int h_s ( x,y )h_t ( y,z ) \mu
( \mathrm{d}y )
\]
are satisfied.
\end{longlist}
Let us denote by $B_\delta  ( y )$ the ball of radius
$\delta$
centered at $y$ and
\[
\mc{P}_{x,t}= \bigl\{ y\in S\dvt h_t ( x,y )>0 \bigr\}.
\]

\begin{theorem}[(Chaumont and Uribe~Bravo \cite{MR2789508})]
\label{WeaklyContinuousBridgesTheorem}
Under \defin{(AC),(C)} and \defin{(CK)}, the law of $X$ on $[0,t]$
under $\mathbf{P}_x$ given $X_t\in B_\delta  ( y )$
converges weakly
in the Skorohod $J_1$ topology to a measure $\mathbf{P}_{x,y}^t$ for
every $y\in\mc{P}_{x,t}$. Furthermore:
\begin{enumerate}
\item The family $ \{ \mathbf{P}_{x,y}^t\dvt y\in\mc{P}_{x,t}
\}$ is a
regular conditional distribution for $X$ on $[0,t]$ given $X_t$ under
$\mathbf{P}_x$.
\item The finite-dimensional distributions of $\mathbf{P}_{x,y}^t$ are
given by
\begin{eqnarray*}
&&\mathbf{P}_{x,y}^t ( X_{t_1}\in
\mathrm{d}x_1,\ldots, X_{t_n}\in \mathrm{d}x_n )
\\
&&\quad=h_{t_1} ( x,x_1 )h_{t_2} ( x_1,x_2
)\cdots h_{t_n-t_{n-1}} ( x_{n-1},x_n )\frac{h_{t-t_n}
( x_n,y )}{h_t  ( x,y )} \,
\mathrm{d}x_1\cdots \,\mathrm{d}x_n.
\end{eqnarray*}
\item As $y'\to y$ and $x'\to x$, $\mathbf{P}_{x',y'}^t$ converges
weakly to $\mathbf{P}_{x,y}^t$.
\end{enumerate}
\end{theorem}
%
\begin{remark*}
The finite-dimensional distributions of the bridge laws can be written
succinctly using the following local absolute continuity condition
valid for $s<t$:
%
\begin{equation}
\label{bridgeLAC} \mathbf{P}_{x,t}^t|_{\F_s}=
\frac{h_{t-s}  ( X_s,y
)}{h_t  ( x,y )}\cdot\mathbf{P}_x|_{\F_s}.
\end{equation}
The reasoning in Revuz and Yor \cite{MR1725357}, Chapter~VIII, implies that for any
stopping time $T$ taking values in $[0,t)$:
%
\begin{equation}
\label{bridgeLACStoppingTime} \mathbf{P}_{x,t}^t|_{\F_T}=
\frac{h_{t-T}  ( X_T,y
)}{h_t  ( x,y )}\cdot\mathbf{P}_x|_{\F_T}.
\end{equation}
\end{remark*}

\section{L\'evy processes and their bridges}
\label{LPSection}
Let $\p_x$ be the Markovian family of a L\'evy process which satisfies
assumptions \defin{(K)} and \defin{(R)}.
As argued by Kallenberg \cite{MR628873}, Fourier inversion implies that $X_t$
possesses continuous and bounded densities which vanish at infinity (by
the Riemann-Lebesgue lemma) for all $t>0$. Actually, \defin{(K)} also
implies that the continuous version $f_t$ of the density of $X_t$ under
$\p_0$ satisfies a form of the Chapman--Kolmogorov equations:
%
\begin{equation}
\label{CKEquations} f_t ( x )=\int f_s ( y
)f_{t-s} ( x-y ) \,\mathrm{d}y \qquad\mbox{for $0<s<t$}.
\end{equation}
From $f_t$ one can build a bi-continuous transition density $p_t$ by
means of $p_t  ( x,y )=f_t  ( y-x )$ which
satisfies \defin{(AC)},
\defin{(C)} and \defin{(CK}).

Under hypotheses \defin{(K)} and \defin{(R)}, Sharpe \cite{MR0240850} shows
that $f_t$ is strictly positive for all $t>0$, which implies that $p_t>0$.

From Theorem \ref{WeaklyContinuousBridgesTheorem}, we see that under
\defin{(K)} and \defin{(R)}, we can consider the bridges $\p_{x,y}^t$
from $x$ to~$y$ of length $t$ for any $x,y\in\re$ and any $t>0$, and
that these are jointly weakly continuous in $x$ and~$y$.

\section{L\'evy processes conditioned to stay positive and meanders}
\label{ConditionedLPAndMeanderSection}
The most general construction for L\'evy processes conditioned to stay
positive, now recalled, is from Chaumont and Doney \cite{MR2164035} (see Chaumont and Doney \cite{MR2375597}
for a correction and Doney \cite{MR2320889} for a lecture note presentation).
When the initial state is positive, it is a Doob transformation of~$\p_x$ by a procedure we now detail. Let
\[
\underline X_t=\min_{s\leq t}X_s
\]
and consider the Markov process $R=X-\underline X$. Under \defin{(R)},
$0$ is regular state of $R$ for itself and so we can consider the local
time at zero $L$ of $R$. We can then define the downwards ladder height
process $H$ of $X$ by
\[
H=-X\circ L,
\]
which is a (possibly killed) subordinator (cf. Bertoin \cite{MR1406564} or
Doney \cite
{MR2320889}). Let $h$ be the renewal function of $H$ given by
\[
h ( x )=\se \biggl( \int_0^\infty
\si_{H_s\leq x} \,\mathrm{d}s \biggr).
\]
For $x>0$, let $\q_x$ be the law of $x+X$ under $\p$ killed when it
leaves $(0,\infty)$, which is a Markov process on $(0,\infty)$ whose
semigroup is denoted $Q= ( Q_t,t\geq0 )$. Chaumont and Doney \cite{MR2164035} prove
that if $X$ drifts to $-\infty$ ($\lim_{t\to\infty}X_t=-\infty$ almost
surely) then $h$ is excessive and otherwise $h$ is invariant for $Q_t$
and proceed to define the semigroup $P^{\uparrow}_t$ by
\[
P_t^{\uparrow} ( x,\mathrm{d}y )=\frac{h  ( y )}{h
( x )}Q_t
( x,\mathrm{d}y )\qquad \mbox{for $x>0$}.
\]
The Markovian laws $\p_x^\uparrow,x>0$, define the L\'evy process
conditioned to stay positive. Note that~$X$ has finite lifetime under
$\p_x^\uparrow$ if and only if $X$ drifts to $-\infty$ under $\p_x$.
Under hypothesis \defin{(R),} Chaumont and Doney \cite{MR2164035} prove that $\p^{\uparrow
}_x$ has a weak limit (in the Skorhod $J_1$ topology) as $x\to0$,
denoted $\p^\uparrow_0$, and that $ ( \p^\uparrow_x
)_{x\geq
0}$ is
Markovian and has the Feller property.

We now give an alternate definition of the meander, from which one can
justify the weak limit construction we have alluded to (cf. Chaumont and Doney \cite{MR2663630}, Lemma 4).
A L\'evy meander is a stochastic process whose law $\p^{\mathrm{me},t}$ satisfies the following absolute continuity relationship with
respect to the law $\p_0^\uparrow$ on $\F_t=\sigma  (
X_s\dvt s\leq t )$:
\[
\p^{\mathrm{me},t}\rrvert_{\F_t}=
\frac{1}{\beta_th  (
X^\uparrow_t )}\cdot\p_0^\uparrow\rrvert_{\F_t}
\qquad\mbox{with } \beta_t= \se^\uparrow_0
\biggl( \frac{1}{h  ( X^\uparrow_t )} \biggr).
\]

\section{Bridges of L\'evy processes conditioned to stay positive}
\label{BridgesOfConditionedLPSection}
We now construct the bridges of a L\'evy process conditioned to stay
positive under hypotheses~\defin{(K)} and \defin{(R)}. This is done
through Theorem \ref{WeaklyContinuousBridgesTheorem} by verifying the
existence of a continuous version of their densities (cf. Lemma \ref
{ExistenceOfDensitiesForLPCSP}). For positive arguments, the density is
constructed from the density of the killed L\'evy process, a continuous
version of which is constructed using bridges of the L\'evy process
itself in Lemma \ref{ExistenceOfDensitiesForKilledProcess}. Then, a
delicate point is to study the densities at $0$; this requires the
following duality lemma.
Let $\hat\p_x$ be the law of $x-X$ and let $\hat P_t$ be its semigroup.
We can also consider the objects $\hat h$, $\hat\p^\uparrow_x$,
etc\ldots associated with $-X$ instead of $X$ as well as $\hat p$.
%
\begin{lemma}
\label{DualityForLPSP}
The semigroups $P_t^\uparrow$ and $\hat P_t^\uparrow$ are in duality
with respect to the measure $\lambda^\uparrow$ given by
\[
\lambda^\uparrow ( \mathrm{d}x )=h ( x )\hat h ( x ) \,\mathrm{d}x.
\]
\end{lemma}
\begin{pf}
Since $P_t$ and $\hat P_t$ are in duality with respect to Lebesgue
measure $\lambda$, it follows that $Q_t$ and $\hat Q_t$ are also in
duality with respect to $\lambda$.

Hence, we get
\begin{eqnarray*}
\int f P^\uparrow_t ( g ) \,\mathrm{d}\lambda^\uparrow &=&
\int f \frac{Q_t  ( gh )}{h} h\hat h \,\mathrm{d}\lambda =\int\hat Q_t ( f
\hat h )g h \,\mathrm{d}\lambda =\int\hat P^\uparrow_t ( f ) g h
\hat h \,\mathrm{d}\lambda \\
&= &\int\hat P^\uparrow_t ( f ) g \,
\mathrm{d}\lambda^\uparrow.
\end{eqnarray*}
\upqed\end{pf}

We now consider the absolute continuity of the semigroup of $X$ killed
when it becomes negative.
%
\begin{lemma}
\label{ExistenceOfDensitiesForKilledProcess}
Under \defin{(K)} and \defin{(R)}, let
%
\begin{equation}
\label{HuntsFormulaForDensity} 
q_t ( x,y )=
\se_{x,y}^t ( \underline X_t>0 )p_t (
x,y ) \qquad\mbox{for $x,y>0$.}
\end{equation}
Then $q_t$ is a transition density for $Q_t$ with respect to Lebesgue
measure which is continuous, strictly positive, bounded by $p$,
satisfies the Chapman--Kolmogorov equations, and which satisfies the
following duality formula:
\[
q_t ( x,y )=\hat q_t ( y,x ).
\]
\end{lemma}
%
\begin{remark*}
It is simple to see that the absolute continuity of $P_t  (
x,\cdot  )$ translates into absolute continuity of $Q_t  (
x,\cdot )$ since if
$A$ has Lebesgue measure zero then
\[
Q_t ( x,A )=\p_x ( X_t\in A, \underline
X_t>0 )\leq\p_x ( X_t\in A )=0.
\]
What is more difficult, is to see that the $q$ is strictly positive;
similar results have been obtained in the literature for killed
(multidimensional) Brownian motion and stable L\'evy processes in Hunt~\cite{MR0079377},
Chen and Song \cite{MR1473631},
Vondra{\v{c}}ek \cite{MR1918106}. Our proof of uses the weakly
continuous Markovian bridges provided by Theorem \ref
{WeaklyContinuousBridgesTheorem}. The almost sure positivity of $q$ can
also be obtained from Theorem 4 of Pitman and Uribe~Bravo \cite{10113069}.
\end{remark*}
\begin{pf*}{Proof of Lemma \ref{ExistenceOfDensitiesForKilledProcess}}
Conditioning on $X_t$, we see that
\begin{eqnarray*}
\se_x \bigl( \si_{\underline X_t>0}f ( X_t ) \bigr) &=&
\se_x \bigl[\p_{x,X_t}^t ( \underline
X_t>0 )f ( X_t ) \bigr]
\\
&=&\int
\p_{x,y}^t ( \underline X_t>0
) f ( y )%
p_t ( x,y ) \lambda ( \mathrm{d}y )
\end{eqnarray*}
for measurable and bounded $f$. On the other hand, the definition of
the law $\q_t$ gives
\[
\se_x \bigl( \si_{\underline X_t>0}f ( X_t ) \bigr)=
\q_x \bigl( f ( X_t ) \bigr)=\int f ( y ) Q_t
( x,\mathrm{d}y )
\]
so that $q$ is a transition density of $Q$ with respect to Lebesgue measure.

We know that $p$ is continuous. To see that $q$ is continuous, it
suffices to apply the portemanteau theorem. Note that the boundary
$\partial \{ \underline X_t>0 \}$ of $ \{ \underline
X_t >0 \}$ has
$\p_{x,y}^t$-measure zero. Indeed, since the minimum on $[0,t]$ is a
continuous functional on Skorohod space (cf. Whitt \cite{MR1876437}, Section 13.4):
\[
\partial \{ \underline X_t>0 \}\subset \bigl\{ X_s\geq 0
\mbox{ for all }s\in[0,t]\mbox{ and there exists }s\in[0,t]\mbox{ such that
}X_s=0 \bigr\}. 
\]
Since $x,y>0$ and under $\p_{x,y}^t$ we have $X_{0+}=x$ and $X_{t-}=y$
almost surely, we see that the process cannot touch zero at times $0$
or $t$. However, using the local absolute continuity relationship~\eqref
{bridgeLACStoppingTime} at the first time $T$ such that $X_{T}=0$, we
see that $\p_{x,y}^t  ( \partial \{ \underline X_t>0
\}=0 )$ as
soon as
\[
\p_x \bigl( \mbox{Touching zero on $(0,t)$ and staying nonnegative}
\bigr)=0,
\]
which is true since $0$ is regular for $(-\infty,0)$.

To prove the duality formula for $q$, we first Proposition II.1 of
Bertoin \cite
{MR1406564}, which proves that $p_t  ( x,y )=\hat p_t
( y,x )$ for
almost all $x$ and $y$ and remove the almost all qualifier by
continuity. Next, Corollary II.3\vspace*{-1pt} of Bertoin \cite{MR1406564} proves that for
almost all $x$ and $y$ the image of $\p_{x,y}^t$ under the time
reversal operator is $\hat\p_{y,x}^t$, which by weak continuity of
bridge laws can be extended to every $x$ and $y$. Since the event
$\underline X_t$ is invariant under time reversal, we see that
\[
q_t ( x,y )=\se_{x,y}^t ( \underline
X_t>0 )p_t ( x,y )=\hat\se_{y,x}^t (
\underline X_t>0 )\hat p_t ( y,x )=\hat q_t (
y,x ).
\]

By definition, we see that $q\leq p$ almost everywhere, and so
continuity implies that $q$ is bounded by $p$ everywhere; this will
help us prove that $q$ satisfies the (CK) equations. Indeed, the Markov
property implies that
%
\begin{equation}
\label{AlmostSureCKForq} q_{t+s} ( x,z )=\int q_s ( x,y
)q_t ( y,z ) \lambda ( \mathrm{d}y )\qquad\mbox{for $\lambda$-almost all
$z$}.
\end{equation}
Since
\[
0\leq q_s ( x,y )q_t ( y,z )\leq p_s ( x,y
)p_t ( y,z )
\]
and
\[
\int p_s ( x,y )p_t ( y,z ) \lambda ( \mathrm{d}z
)=p_{t+s} ( x,z ),
\]
which is continuous in $z$, the generalized dominated convergence
theorem tells us that
\[
z\mapsto\int q_s ( x,y )q_t ( y,z ) \lambda (
\mathrm{d}y )
\]
is continuous (on $(0,\infty)$). Because both sides of \eqref
{AlmostSureCKForq} are continuous, we can change the almost sure
qualifier to for all $z$.

It remains to see that $q_t  ( x,y )>0$ if $x,y,t>0$. 
We first prove that for any $x,y,t>0$, if $\delta>0$ is such that
$B_{\delta}  ( y )\subset(0,\infty)$, then
%
\begin{equation}
\label{KilledSemigroupHasFullSupport} Q_t \bigl( x,B_{\delta}
( y ) \bigr)>0.
\end{equation}
This is done by employing a technique of Knight \cite{MR1417982}. For any
$s\in(0,t)$, consider the process $ ( X^s_r,r\leq t )$
given by
\[
X^s_r= X_0+X_{ ( r+s )\bmod t}-X_s.
\]
Since $X$ has independent and stationary increments, then, for any
fixed $s$, the laws of $X^s$ and $ ( X_r,r\leq t )$ coincide under
$\p_x$ for any $x\in\re$; this is referred to as the cyclic
exchangeability property in Chaumont, Hobson and Yor \cite{MR1837296}. Note that $X^s_t=X_t$; if
$s$ is close to the place where~$X$ reaches its minimum on $(0,t)$,
then the minimum $\underline X^s_t$ of $X^s$ on the interval $[0,t]$ is
positive. Hence, the random variable
\[
I=\int_0^t \si_{\underline X^s_t>0, X_t\in B_{\delta}  (
y )} \,\mathrm{d}s
\]
is positive on $ \{ X_t\in B_{\delta}  ( y ) \}
$ which has positive
probability since $p_t$ is strictly positive. On the other hand, from
cyclic exchangeability, we can compute:
\[
0<\se_x ( I ) =\int_0^t
\p_x \bigl[\underline X^s_t>0,
X^s_t\in B_{\delta } ( y ) \bigr] \,\mathrm{d}s =t
\p_x \bigl[\underline X_t>0, X_t\in
B_{\delta} ( y ) \bigr] =tQ_t \bigl( x,B_{\delta} ( y )
\bigr),
\]
which proves \eqref{KilledSemigroupHasFullSupport}. To prove positivity
of $q_t$, first note that since $\hat\p_y$ almost surely $X_{0+}=y$,
then for $s$ small enough:
\[
\hat Q_s \bigl( y,B_{\delta} ( y ) \bigr) =\hat
\p_y \bigl( X_s\in B_{\delta} ( y ),\underline
X_s>0 \bigr)\geq\hat \p_y \bigl( X_r\in
B_{\delta} ( y )\mbox{ for all }r\in[0,s] \bigr)>0,
\]
so that, by continuity of $q_s$, there exists an open subset $U_s$ of
$B_{\delta}  ( y )$ such that $q_s  ( \cdot,y
)=\hat q_s  ( y,\cdot  )>0$ on $U_s$. By
Chapman--Kolmogorov and \eqref
{KilledSemigroupHasFullSupport}, we see that
\[
q_t ( x,y )\geq\int_{U_s} Q_{t-s} ( x,
\mathrm{d}z )q_s ( z,y )>0.
\]
\upqed\end{pf*}

We now turn to a similar result for L\'evy processes conditioned to
stay positive.
%
\begin{lemma}
\label{ExistenceOfDensitiesForLPCSP}
Under \defin{(K)} and \defin{(R)},
$P^\uparrow_t  ( x,\cdot )$ is equivalent to Lebesgue
measure for all
$t>0$ and $x\geq0$. Furthermore, there exists a version of the
transition density $p^\uparrow$ which is continuous, strictly positive,
and satisfies the Chapman--Kolmogorov equations.
\end{lemma}
Therefore, the density $p^\uparrow$ satisfies the assumptions \defin
{(AC), (C)} and \defin{(CK)} of Theorem~\ref{WeaklyContinuousBridgesTheorem}.
\begin{pf*}{Proof of Lemma \ref{ExistenceOfDensitiesForLPCSP}}
Since the renewal function of a subordinator is positive, continuous
and finite, we deduce by $h$-transforms and Lemma \ref
{ExistenceOfDensitiesForKilledProcess} that the function
\[
p^\uparrow_t ( x,y )=\frac{q_t  ( x,y )}{h
( x )\hat h  ( y )},\qquad  x>0,y>0,t>0,
\]
is a transition density for $P^\uparrow$ starting at positive states:
\[
P^\uparrow_tf ( x )=\int p^\uparrow_t ( x,y
)f ( y ) \lambda^\uparrow ( \mathrm{d}y ) \qquad\mbox{for $x>0$.}
\]
Notice that $p^\uparrow$ so defined is strictly positive, continuous,
and satisfies the Chapman--Kolmogorov equations.

For $0<s<t$, consider the function
\[
p^{\uparrow s}_t ( y )=\int P^\uparrow_s ( 0,
\mathrm{d}x )p^\uparrow_{t-s} ( x,y )>0 \qquad\mbox{for } y>0.
\]
On one hand, Chapman--Kolmogorov implies that for any bounded measurable $f$:
\[
\int p^{\uparrow s}_t ( y )f ( y ) \lambda^\uparrow (
\mathrm{d}y )=\int\!\!\int P^\uparrow_s ( 0,\mathrm{d}x
)p^\uparrow_{t-s} ( x,y )f ( y ) \lambda^\uparrow (
\mathrm{d}y )=\int P^{\uparrow }_t ( 0,\mathrm{d}y )f ( y ),
\]
so that $p^{\uparrow s}_t$ is a version of the density of $P_t^\uparrow
( 0,\cdot )$ with respect to $\lambda^\uparrow$ and so if
$0<s<s'<t$ then $p_t^{\uparrow s}  ( y )=p_t^{\uparrow s'}
( y )$
for $\lambda$-almost all $y$. On the other hand, we now see that
$p_t^{\uparrow s}  ( y )$ is a continuous function of $y$,
so that
actually, the almost sure qualifier can be dropped. Indeed, since
\[
M_{t-s}:=\sup_{x,y}p_{t-s} ( x,y )<\infty
\]
and from Chaumont and Doney \cite{MR2164035}
\[
\beta_s:=\p_0^\uparrow \bigl( 1/h (
X_s ) \bigr)<\infty,
\]
continuity of $p^{\uparrow s}_t$ follows from the dominated convergence theorem.

We can now define
\[
p^\uparrow_t ( 0,y )=p^{\uparrow s}_t ( y ),\qquad y>0,
\]
for any $s\in(0,t)$. Since $p^\uparrow_t  ( 0,y )$ is
continuous for
$y\in(0,\infty)$, and is a version of the density of $P^\uparrow_t
( 0,\cdot )$,
the Markov property implies:
\[
p^\uparrow_{t+s} ( 0,y )=\int p^\uparrow_t (
0,x )p^\uparrow_s ( x,y ) \lambda^\uparrow ( \mathrm{d}x
).
\]
Furthermore, we have the bound
\[
p^\uparrow_t ( 0,y )\leq\beta_s M_{t-s}/
\hat h ( y )\qquad \mbox{for }y>0.
\]
We now provide an uniform bound for the transition density in the
initial state. Recall that $p_t  ( x,y )\to0$ as $x\to
\infty$. Since
$q_t\leq p_t$, then $p_t^\uparrow  ( x,y )\to0$ as $x\to
\infty$ for
any $t>0$ and $y>0$. Choose now any $s\in(0,t)$. By Chapman--Kolmogorov:
\begin{eqnarray*}
p^{\uparrow}_t ( x,z ) &=&\int p^\uparrow_s (
x,y )p^{\uparrow}_{t-s} ( y,z ) \lambda^\uparrow ( \mathrm{d}y
)
\\
&\leq&\int p^\uparrow_s ( x,y )\frac{M_{t-s}}{h
( y )\hat h  ( z )}
\lambda^\uparrow ( \mathrm{d}y ) 
\\
&\leq&\se_x^\uparrow \biggl( \frac{1}{h  ( X_s
)} \biggr)
\frac
{M_{t-s}}{\hat h  ( z )}.
\end{eqnarray*}
Note that $x\mapsto\se_x^\uparrow  ( \frac{1}{h  (
X_s )} )$ is
continuous on $(0,\infty)$, hence bounded on compact subsets of
$(0,\infty)$. Continuity at zero is proved in Corollary 1 of Chaumont and Doney \cite
{MR2375597}. Hence, we obtain
\[
\sup_{x\geq0}p_t^\uparrow ( x,y )<\infty\qquad\mbox{for all
$y>0$ and $t>0$}.
\]

We now prove that
%
\begin{equation}
\label{ContinuityAssertionForDensityOfLPCSP} \lim_{x\to0}p^\uparrow_t
( x,y )=p^\uparrow_t ( 0,y ) \qquad\mbox {for }y>0.
\end{equation}
Indeed, from the Chapman--Kolmogorov equations
\[
p^\uparrow_{t} ( x,z )=\int p^\uparrow_{t-s} (
y,z ) P^\uparrow_s ( x,\mathrm{d}y ).
\]
Note that $P_s^\uparrow  ( x,\cdot )$ converges weakly to
$P_s^\uparrow  ( 0,\cdot )$ as $x\to0$ and that
$p^\uparrow_{t-s}  ( \cdot,z )$ is continuous and
bounded on $(0,\infty)$, which is
the support of $P_s^{\uparrow}  ( 0,\cdot )$.

By applying the above arguments to the dual process, we can define
$p^\uparrow_t  ( x,0 )$ as $\hat p^\uparrow_t  (
0,x )$ and note that
\[
\lim_{y\to0}p^\uparrow_t ( x,y )=p^\uparrow_t
( x,0 ) \qquad\mbox {for }x>0.
\]
We can now define
\[
p_t^{\uparrow,s} ( 0,0 )=\int p^\uparrow_s ( 0,y
)p^\uparrow_{t-s} ( y,0 ) \lambda^\uparrow ( \mathrm{d}y ).
\]
To show that the above definition does not depend on $s$, we now show that
$\lim_{z\to0}p^\uparrow_t  ( 0,z )=p_t^{\uparrow ,s}
( 0,0 )$. By
Chapman--Kolmogorov, we get
\[
p^\uparrow_t ( 0,z )=\int p^\uparrow_s ( 0,y
)p^\uparrow_{t-s} ( y,z ) \lambda^\uparrow ( \mathrm{d}y ).
\]
We know that $p^\uparrow_{t-s}  ( y,z )$ converges to
$p^\uparrow_{t-s}  ( y,0 )$ as $z\to0$. Dominated
convergence, which applies because
of the bound
\[
p^\uparrow_{t-s} ( y,z )\leq C/h ( y ),
\]
then implies
\[
\lim_{z\to0}p^\uparrow_t ( 0,z )=\int
p^\uparrow_s ( 0,y )p^\uparrow_{t-s} ( y,0 )
\lambda^\uparrow ( \mathrm{d}y ),
\]
which shows that we can define $p_t^\uparrow  ( 0,0
)=p_t^{\uparrow,s}  ( 0,0 )$, and we have
\[
\lim_{z\to0}p^\uparrow_t ( 0,z )=p^\uparrow_t
( 0,0 )\qquad\mbox {and by duality}\quad\lim_{x\to0}p^\uparrow_t
( x,0 )=p^\uparrow_t ( 0,0 ).
\]

Finally, we will prove that
\[
\lim_{x,z\to0}p^\uparrow_t ( x,z )=p_t ( 0,0
).
\]
Take $x_n,z_n\to0$ and write
\begin{eqnarray*}
&&\limsup_n \bigl\llvert p^\uparrow_t (
x_n,z_n )-p^\uparrow_t ( 0,0 ) \bigr
\rrvert \\
&&\quad\leq\limsup_n \bigl\llvert p^\uparrow_t (
x_n,z_n )-p^\uparrow_t (
x_n,0 ) \bigr\rrvert +\limsup_n \bigl\llvert
p^\uparrow_t ( x_n,0 )-p^\uparrow_t
( 0,0 ) \bigr\rrvert
\\
&&\quad\leq\limsup_n\int P^\uparrow_s ( x_n,
\mathrm{d}y ) \bigl\llvert p^\uparrow_{t-s} ( y,z_n
)-p^\uparrow_{t-s} ( y,0 ) \bigr\rrvert .
\end{eqnarray*}
Since $P^\uparrow_s  ( x_n,\cdot )$ weakly to
$P^\uparrow_s  ( 0,\cdot )$ and
\[
\bigl\llvert p^\uparrow_{t-s} ( y,z_n
)-p^\uparrow_{t-s} ( y,0 ) \bigr\rrvert \leq C/h ( y ),
\]
where $C$ is a finite constant, we obtain the desired result.
\end{pf*}
The main result of this section is the construction of weakly
continuous bridges for the L\'evy process conditioned to stay positive.
Indeed, by applying Theorem \ref{WeaklyContinuousBridgesTheorem} and
Lemma \ref{ExistenceOfDensitiesForLPCSP}, we obtain Theorem \ref
{WeaklyContinuousBridgeLPCSP}.

The proof of Corollary \ref{DIMCorollary} is simple from Theorem \ref
{WeaklyContinuousBridgeLPCSP} and the following remarks. First, we
note that the finite-dimensional distributions of the bridges $\p_{x,y}^{\uparrow,t}$ and $\q_{x,y}^t$ are identical if $x,y,t>0$
(because we have an $h$-transform relationship between $\q_x$ and $\p^\uparrow_x$ for $x>0$). Next, note that the law of $X-\eps$ under
$\q_{\eps,\eps}^{t}=\p^{\uparrow,t}_{\eps,\eps}$ is precisely that
of $\p_{0,0}^t$ conditioned on $\underline X_t>-\eps$. Finally, since the
laws $\p_{x,y}^{\uparrow, t}$ are weakly continuous, Corollary \ref
{DIMCorollary} is established.

\section{An extension of the Denisov decomposition of the Brownian trajectory}
\label{DenisovSection}
We now turn to the extension of the Denisov decomposition of the
Brownian trajectory of Theorem \ref{DenisovTheorem}.
\begin{pf*}{Proof of Theorem \ref{DenisovTheorem}}
We will use Lemma 4 in Chaumont and Doney \cite{MR2663630}, which states that if $x_n\to
0$ and $t_n\to t>0$ then the law of $ ( X_s,s\leq t_n )$
conditionally on $\underline X_{t_n}>0$ under $\p_{x_n}$ converges as
$n\to\infty$ in the sense of finite-dimensional distributions to $\p^{\mathrm{me},t}$ when $0$ is regular for $(0,\infty)$. (This was only stated in
Chaumont and Doney \cite{MR2663630} for fixed $t$ and follows from Corollary 2 in Chaumont and Doney \cite
{MR2164035}. However, the arguments, which are actually found in Chaumont and Doney \cite
{MR2375597}, also apply in our setting.)

Fix $t>0$. Since $0$ is regular for both half-lines, $X$ reaches its
minimum $\underline X_t$ on the interval $[0,t]$ continuously at an
unique place $\rho_t$, as proved in Propositions 2.2 and 2.4 of Millar \cite
{MR0433606}.
Let
\[
\rho_t^n= \bigl\lfloor\rho_t 2^n
\bigr\rfloor/2^n
\]
and note that
\[
\underline X_t=\min_{s\in[\rho_t^n, \rho_t^n+1/2^n]}X_s.
\]
For continuous and bounded $f\dvt \re\to\re$ and functions $F$ of $G$
of the form $h  ( X_{t_1},\ldots, X_{t_m} )$ for some
$t_1,\ldots
,t_m\geq0$ and continuous and bounded $h$, we will compute the quantity
\[
\se_0 \bigl( F ( X_{\cdot\wedge\rho_n} )f ( \rho_n )G (
X_{ ( \rho_n+1/2^n+\cdot )\wedge
t}-X_{\rho_n} ) \bigr).
\]
This is done by noting the decomposition
\[
\bigl\{ \rho_t^n=k/2^n \bigr
\}=A_{k,n}\cap B_{k,n},
\]
where
\begin{eqnarray*}
A_{k,n}&=& \bigl\{ m_{k,n} \leq X_s\mbox{ for }s
\leq k/2^n \bigr\},
\\
B_{k,n}&=& \bigl\{ m_{k,n} \leq X_s\mbox{ for }s\in
\bigl[ ( k+1 )/2^n,t \bigr] \bigr\}
\end{eqnarray*}
and
\[
 m_{k,n} =\inf_{r\in[ k/2^n, ( k+1 )/2^n]}X_r.
\]
Applying the Markov property at time $(k+1)/2^n$ we obtain
\begin{eqnarray*}
&&\se_0 \bigl( F ( X_{\cdot\wedge\rho_n} )f ( \rho_n )G (
X_{ ( \rho_n+1/2^n+\cdot )\wedge
t}-m_{k,n} )\si_{\rho_n=k/2^n} \bigr)
\\
&&\quad=\se_0 \bigl( F ( X_{\cdot\wedge\rho_n} )f ( \rho_n ) H
\bigl( t- ( k+1 )/2^n,X_{ (
k+1 )/2^n},\underline X_{ ( k+1 )/2^n}
\bigr)\si_{A_{k,n}} \bigr),
\end{eqnarray*}
where
\[
H ( s,x,y )=\se_x \bigl( G \bigl( X^{s}-y \bigr)
\si_{\underline X_s>y} \bigr)=\se_{x-y} \bigl( G \bigl( X^{s}
\bigr)\si_{\underline X_s>0} \bigr).
\]
By reversing our steps, we obtain
\begin{eqnarray*}
&&\se_0 \bigl( F ( X_{\cdot\wedge\rho_n} )f ( \rho_n ) H
\bigl( t- ( k+1 )/2^n,X_{ (
k+1 )/2^n},\underline X_{ ( k+1 )/2^n}
\bigr)\si_{A_{k,n}} \bigr)
\\
&&\quad= \se_0 \bigl( F ( X_{\cdot\wedge\rho_n} )f ( \rho_n )
\tilde H \bigl( t- ( k+1 )/2^n,X_{ (
k+1 )/2^n},\underline
X_{ ( k+1 )/2^n} \bigr)\si_{\rho_t^n=k/2^n} \bigr),
\end{eqnarray*}
where
\[
\tilde H ( s,x,y )=\se_{x-y} \bigl( \cond{G \bigl( X^{s}
\bigr)} {\underline X_s>0} \bigr).
\]
By the continuity assumptions of $f,F$ and $G$ we can pass to the limit
using the Chaumont--Doney lemma to get
\[
\se_0 \bigl( F ( X_{\cdot\wedge\rho_t} )f ( \rho_t )G
\bigl( {X}^{\rightarrow} \bigr) \bigr) =\se_0 \bigl[F (
X_{\cdot\wedge\rho_t} )f ( \rho_t )\se^{\mathrm{me},t-\rho_t} ( G ) \bigr].
\]

By time reversal at $t$, we see that
\[
\se_0 \bigl( F \bigl( {X}^{\leftarrow} \bigr)f (
\rho_t )G \bigl( {X}^{\rightarrow} \bigr) \bigr) = \se_0
\bigl[\hat\se^{\mathrm{me},\rho_t} ( F )f ( \rho_t ) \se^{\mathrm{me},t-\rho_t} ( G
) \bigr].
\]
\upqed\end{pf*}

We now establish a Denisov-type decomposition for bridges of L\'evy processes.

\begin{pf*}{Proof of Theorem \ref{DenisovForBridgesTheorem}}
Since $0$ is regular for $(-\infty,0)$ under $\p_0$, using local
absolute continuity between $\p_{0,0}^t$ and $\p_0$ we see that
$\underline X_t<0$ and $\rho_t>0$ almost surely under $\p_{0,0}^t$.
Time reversal and regularity of $0$ for $(0,\infty)$ proves that $\rho_t<t$ almost surely under $\p_{0,0}^t$.

From the absolute continuity relationship between the meander and the
L\'evy process conditioned to stay positive, we see that $\p_{0,x}^{\uparrow,t}$ is a regular conditional probability of $X$ given
$X_t=x$ under $\p^{\mathrm{me},t}$. Hence, Theorem \ref{DenisovTheorem} allows
the conclusion
\[
\se_0 \bigl( F_1 \bigl( {X}^{\leftarrow} \bigr)f (
\rho_t )g ( \underline X_t,X_t
)F_2 \bigl( {X}^{\rightarrow} \bigr) \bigr) =\se_0
\bigl[\hat\se_{0,-\underline X_t}^{\uparrow,\rho _t} ( F_1 )f (
\rho_t )g ( \underline X_t,X_t )\hat
\se_{0,X_t-\underline X_t}^{\uparrow,\rho_t} ( F_2 ) \bigr].
\]
Hence, we see that for every continuous and bounded $f,g_1,g_2,F_1,F_2$:
\begin{eqnarray*}
&&\int\se_{0,x}^{t} \bigl[F_1 \bigl(
{X}^{\leftarrow} \bigr)f ( \rho_t )g_1 ( \underline
X_t )F_2 \bigl( {X}^{\rightarrow} \bigr)
\bigr]g_2 ( x )p_t ( 0,x ) \,\mathrm{d}x
\\
&&\quad=\int\se_{0,x}^{t} \bigl[\hat\se_{0,\underline X_t}^{\uparrow
,\rho_t}
( F_1 )f ( \rho_t )g_1 ( \underline
X_t )\se_{0,\underline X_t}^{\uparrow,\rho_t} ( F_2 )
\bigr]g_2 ( x )p_t ( 0,x ) \,\mathrm{d}x.
\end{eqnarray*}
Since both integrands are continuous because of weak continuity of the
bridge laws (of the L\'evy process, its dual, and their conditioning to
remain positive), we see that
\begin{eqnarray*}
\se_{0,0}^{t} \bigl( F_1 \bigl(
{X}^{\leftarrow} \bigr)f ( \rho_t )g_1 ( \underline
X_t )F_2 \bigl( {X}^{\rightarrow} \bigr) \bigr) =
\se_{0,0}^{t} \bigl[\hat\se_{0,\underline X_t}^{\uparrow ,\rho
_t} (
F_1 )f ( \rho_t )g_1 ( \underline
X_t )\se_{0,\underline X_t}^{\uparrow,\rho_t} ( F_2 ) \bigr].
\end{eqnarray*}
\upqed\end{pf*}

Theorems \ref{DenisovTheorem} and \ref{DenisovForBridgesTheorem} imply
the following corollary.

\begin{corollary}
\label{JointDensityPositionMinimumFinalValueCorollaryWithBridge}
The joint law of $ ( \rho_t,\underline X_t,X_t )$ under $\p_0$
admits the expression
\[
\p_0 ( \rho_t\in \mathrm{d}s,-\underline
X_t\in \mathrm{d}y,X_t-\underline X_t\in
\mathrm{d}z )=\p_0 ( \rho_t\in \mathrm{d}s )\hat
\p^{\mathrm{me},s} ( X_s\in \mathrm{d}y ) \p^{\mathrm{me},t-s} (
X_{t-s}\in \mathrm{d}z ).
\]
The joint law of $ ( \rho_t,\underline X_t )$ under $\p_{0,0}^t$
admits the expression
\[
\p_{0,0}^t ( \rho_t\in \mathrm{d}s, -\underline
X_t\in \mathrm{d}y )=\frac
{\p_0  ( \rho_t\in \mathrm{d}s )}{p_t  ( 0,0 )} \hat\p^{\mathrm{me},s} (
X_s\in \mathrm{d}y ) \p^{\mathrm{me},t-s} ( X_{t-s}\in
\mathrm{d}y ).
\]
%
\end{corollary}
%
%

\section{An extension of Vervaat's theorem}
\label{VervaatSection}
In this section, we prove Theorem \ref{VervaatTheorem}.
\begin{pf*}{Proof of Theorem \ref{VervaatTheorem}}
Let $\lambda_t$ be Lebesgue measure on $(0,t)$; we will work under the
law $\p_{0,0}^{\uparrow,t}\otimes\lambda_t$ and we keep the notation
$X$ for the canonical process (which is now defined on the product
space $\Omega\times(0,t)$) and $U$ will be the projection in the
second coordinate of this space. Then a regular version of the law of
$X_{r\wedge U},r\geq0$, and $X_{ ( U+r )\wedge t},r\geq0$, given
$U=u$ and $X_U=y$ is $\p_{0,y}^{\uparrow,t}\otimes\p_{y,0}^{\uparrow,
t}$ and the law of $ ( U,X_U )$ admits the following density:
\[
( u,y )\mapsto\frac{p^\uparrow_s  ( 0,y
)p^\uparrow_{t-s}  ( y,0 )}{t\cdot p^\uparrow_t  (
0,0 )} \,\mathrm{d}u \, \lambda^\uparrow (
\mathrm{d}x ).
\]

On the other hand, the Vervaat transformation of $X$ is the
concatenation of $X^{\rightarrow}$ followed by the time-reversal of
$X^{\leftarrow}$ at $\rho_t$; under $\p_{0,0}^t$, the joint law of
$ ( X^\rightarrow,X^\leftarrow )$ given $\rho_t=t-s$ and
$\underline
X_t=y$ is $\p_{0,y}^{\uparrow,s}\otimes\p_{y,0}^{\uparrow,t-s}$.

We finish the proof by identifying the law of $ ( t-\rho_t,\underline X_t )$ under $\p_{0,0}^t$ with that of $ (
U,X_U )$
under $\p_{0,0}^{\uparrow,t}\times\lambda_1$. Indeed, by Corollary
\ref
{JointDensityPositionMinimumFinalValueCorollaryWithBridge}, a version
of the density with respect to Lebesgue measure of $ ( t-\rho_t,-\underline X_t )$ at $(s,y)$ is
\[
\frac{\p_0  ( \rho_t\in t-\mathrm{d}s )p^\uparrow_{t-s}  (
0,y )p_{s}^{\uparrow}  ( y,0 )h  ( y
)\hat h  ( y )}{p_t  ( 0,0 )\hat
\beta_{t-s}\beta_{s}}.
\]
However, by the Chapman--Kolmogorov equations we can obtain the marginal
density of $\rho_t$ under $\p_{0,0}^t$ at $u$:
\[
\frac{\p_0  ( \rho_t\in t-\mathrm{d}s )p^\uparrow_t  (
0,0 )}{p_t  ( 0,0 )\hat\beta_{t-s}\beta_s}.
\]
Since $\rho_t$ has an uniform law under $\p_{0,0}^t$ as proved in
Knight \cite
{MR1417982}, then, actually the above expression is almost surely equal
to $1/t$ so that a joint density of $ ( \rho_t,\underline
X_t )$
under $\p_{0,0}^t$ is
\[
( u,y )\mapsto\frac{p^\uparrow_s  ( 0,y
)p^\uparrow_{t-s}  ( y,0 )}{t\cdot p^\uparrow_t  (
0,0 )} \,\mathrm{d}u\, \lambda^\uparrow (
\mathrm{d}y ).
\]
\upqed\end{pf*}

\section*{Note added in proof}
It has been pointed out to the author that the proof of what we state as
Theorem~\ref{WeaklyContinuousBridgesTheorem} (taken from
reference \cite{MR2789508}) has an error. Since we use this theorem to
construct our bridges, the reader should note
that Theorem \ref{WeaklyContinuousBridgesTheorem} has a simple proof when the Markov process
in the statement has a Feller dual. This is
the case both for L\'{e}vy processes and for L\'{e}vy processes
conditioned to stay positive (thanks to Lemma \ref{DualityForLPSP}
for the latter), and this ensures the validity of the
results in this paper.

\section*{Acknowledgements}
The author would like to thank Lo\"ic Chaumont, Pat Fitzsimmons, and
Jim Pitman for stimulating conversations, information, references which
were helpful in developing this research, and for spotting some
obscurities and errors in the paper. Further thanks are due to Jim
Pitman who suggested to look for further extensions of Vervaat's
theorem and performed the work of an engaging postdoctoral supervisor.
Research partially conducted at the Department of Statistics of the
University of California at Berkeley. Supported by a
postdoctoral fellowship from UC MexUS -- CoNaCyt,
NSF Grant DMS-08-06118 and PAPIIT Grant IN100411.

%


\printhistory

\end{document}